\documentclass[12pt,a4paper, notitlepage]{article}
\usepackage[utf8]{inputenc}
\usepackage[english]{babel}

\usepackage{biblatex}
\usepackage{csquotes}
\usepackage{amssymb}
\usepackage{amsmath}
\usepackage{siunitx}
\usepackage{IEEEtrantools}
\usepackage{amsthm}
\newtheorem{theorem}{Theorem}
\newtheorem{definition}{Definition}
\newtheorem{claim}{Claim}
\newtheorem{remark}{Remark}

\DeclareMathOperator{\id}{id}
\DeclareMathOperator*{\im}{im}
\DeclareMathOperator*{\Sp}{Sp}
\DeclareMathOperator*{\Ann}{Ann}
\DeclareMathOperator*{\tc}{tc}
\DeclareMathOperator*{\lcm}{lcm}
\DeclareMathOperator*{\Flat}{Flat}
\DeclareMathOperator*{\dir}{dir}

\DeclareMathOperator*{\modp}{mod}

\begin{document}

\title{The Trajectory Coset and Similarity Classes of Affine Maps}
\author{Arieh Yakir}
\date{May 2020}
\maketitle

\begin{abstract}
In this work we define the trajectory coset of an affine map and use it to study the similarity classes of affine maps. 

The characterization of similarity classes of affine maps, was already accomplished in a previous paper \cite{TartarYakir} but in this work, we use the trajectory coset, a tool which allows us to gain a deeper understanding of the interplay between geometry (properties of affine maps) and algebra (properties of linear maps).

We first state the geometric problem of similarity of affine maps. 

We then develop the algebraic tools. The main idea is the development of an invariant which determines whether one coset can be taken to another coset, under isomorphism of modules.

After resolving this problem we go back to geometrical questions, similarity and invariant flats.
\end{abstract}


\section{The Geometric Set-Up} 
In the following ($\mathcal{E},V$) is an affine space over a field $F$.
This means that we are given a data of:
\begin{enumerate}
  \item $V$ is a vector space over the field $F$.
  \item An action of the group $V$ on the set $\mathcal{E}$
  \[\mathcal{E} \times V\longrightarrow \mathcal{E} \]
  satisfying (on top of the axioms for a group action):
  
  For each $A,B\in \mathcal{E}$, there is a unique $\alpha\in V$ such that
  \[A + \alpha = B \]
  
  ( this $\alpha$ is denoted by $\overrightarrow{AB}$ ).
\end{enumerate}
A function $f:\mathcal{E}\longrightarrow \mathcal{E}$ is called an affine map if there exists a linear map $\overline{f}:V\longrightarrow V$ such that for every $A,B\in \mathcal{E}$
\[\overrightarrow{f(A)f(B)}=\overline{f}(\overrightarrow{AB}) \]
The linear map $\overline{f}:V\longrightarrow V$ is called the linear part of $f$.
An affine map is uniquely determined by its linear part $\overline{f}$ and by its action on a single point $A\in \mathcal{E}$.

For a linear map $L:V\longrightarrow V$ there is a standard construction of the module $V^L$ over the ring $F[X]$. 

The underlying additive group is the additive group of $V$. For each $P(X)\in F[X]$ and each $w\in V$, \[P(X)\cdot w:=P(L)(w). \]
Given a linear map $L:V\longrightarrow V$, a key ingredient of our discussion is the submodule $(X-1)\cdot V^L$. Clearly $(X-1)\cdot V^L=\text{im}(L-\text{id}_V)$.

We define now the trajectory coset of an affine map.

\begin{definition} 
Let $f:\mathcal{E}\longrightarrow \mathcal{E}$ be an affine map. Define the trajectory coset of $f$ to be the set: \[\tc(f):=\left\{\overrightarrow{Af(A)}\;\mid\; A\in\mathcal{E} \right\} \]
\end{definition}
We have to prove that the set $\tc(f)$ is a coset of the submodule $(X-1)\cdot V^{\overline{f}}$.
\begin{claim} 
For each $A,B\in\mathcal{E}$, $$\overrightarrow{Af(A)}-\overrightarrow{Bf(B)}=(\id_V-\overline{f})(\overrightarrow{AB})$$
\end{claim}
\begin{proof} 
\begin{IEEEeqnarray*}{rCl}
\overrightarrow{Af(A)}-\overrightarrow{Bf(B)}& = &
\left ( \overrightarrow{Af(A)}+\overrightarrow{f(A)B}\right )-\left ( \overrightarrow{f(A)B}+\overrightarrow{Bf(B)}\right )\\
& = & \overrightarrow{AB}-\overrightarrow{f(A)f(B)}   \\
& = & \overrightarrow{AB}-\overline{f}\left (\overrightarrow{AB}\right ) \\
& = &\left (\id_V - \overline{f} \right )\left ( \overrightarrow{AB}\right )
\end{IEEEeqnarray*}
\end{proof}

\begin{claim} 
For each $A,B\in \mathcal{E},$
\[ \overrightarrow{Af(A)}\equiv\overrightarrow{Bf(B)}\quad \left (\;\text{mod}\quad(X-1)\cdot V^{\overline{f}} \right )\]
\qed
\end{claim}

\begin{claim} 
Let $A\in\mathcal{E}$ and let $\gamma\in\overrightarrow{Af(A)}+(X-1)\cdot V^{\overline{f}}.$\\
Then there exists a point $B\in \mathcal{E}$, such that $\gamma=\overrightarrow{Bf(B)}$.
\end{claim}
\begin{proof}
Find $\alpha\in V$ such that: $\gamma =\overrightarrow{A f(A)}+\left( \overline{f}-\id_V\right )(\alpha)$.

Define $B$ by $B:=A+\alpha$.
\begin{IEEEeqnarray*}{rCl}
\overrightarrow{AB}& = & \alpha\\
\overline{Af(A)}-\overrightarrow{Bf(B)}& = &\left ( \text{id}_V-\overline{f}\right )(\overrightarrow{AB})\\
\overline{Af(A)}-\overrightarrow{Bf(B)}& = &\left ( \text{id}_V-\overline{f}\right )(\alpha)
\end{IEEEeqnarray*}
\begin{IEEEeqnarray*}{rCl}
\gamma & = &\overrightarrow{Af(A)} - \left ( \text{id}_V-\overline{f}\right )(\alpha)\\
& = & \overrightarrow{Af(A)}-\left ( \overrightarrow{Af(A)}-\overrightarrow{Bf(B)}\right ) = \overrightarrow{Bf(B)}.
\end{IEEEeqnarray*}
\end{proof}

\begin{claim} 
Let $f:\mathcal{E}\longrightarrow\mathcal{E}$ be an affine map.
Fix a point $A\in\mathcal{E}$. Then:
\[\tc(f)=\overrightarrow{Af(A)}+(x-1)\cdot V^{\overline{f}}\].
\end{claim}
\noindent\textit{Proof:} 
This claim follows from the previous two claims.

\section{Similarity of Affine Maps} 
\begin{definition} 
Let $f,g\in\mathcal{E}\longrightarrow\mathcal{E}$ be affine maps.

We say that $f$ and $g$ are similar (and write $f\sim g$) if there exists an affine isomorphism, $h:\mathcal{E}\longrightarrow\mathcal{E}$, such that $h\circ f\circ h^{-1}=g$
\end{definition}

\begin{claim} 
For every affine map $h:\mathcal{E}\longrightarrow\mathcal{E}$, $h$ is an affine isomorphism, if and only if, the linear map $\overline{h}:V\longrightarrow V$ is a linear isomorphism.
\qed
\end{claim}

\begin{theorem} 
Let $f,g:\mathcal{E}\longrightarrow\mathcal{E}$ be affine maps. If $h:\mathcal{E}\longrightarrow\mathcal{E}$ is an affine isomorphism such that $h\circ f\circ h^{-1}=g$, then the linear isomorphism $\overline{h}:V\longrightarrow V$ is an isomorphism of modules over $F[X]$, $\overline{h}: V^{\overline{f}}\longrightarrow V^{\overline{g}}$, it takes the submodule $(X-1)\cdot V^{\overline{f}}$ to the submodule $(X-1)\cdot V^{\overline{g}}$, and the coset $\tc(f)$ to the coset $\tc(g)$.
\end{theorem}

\begin{proof}
\begin{IEEEeqnarray*}{rCl}
h\circ f\circ h^{-1}&=&g\\
h \circ f&=&g \circ h\\
\overline{h\circ f}&=&\overline{g\circ h}\\
\overline{h}\circ\overline{f}&=&\overline{g}\circ\overline{h}
\end{IEEEeqnarray*}

On the left hand side of this equation, $\overline{f}$ stands for the action of $X$ on $V^{\overline{f}}$, while on the right hand side, $\overline{g}$ stands for the action of $X$ on $V^{\overline{g}}$.

Therefore the equation states that $\overline{h}\cdot X=X\cdot \overline{h}$.

Therefore $\overline{h}$ commutes with every polynomial in $F[X]$.

Since $\overline{h}$ commutes with $X-1$, it takes the submodule $(X-1)\cdot V^{\overline{f}}$ to the submodule $(X-1)\cdot V^{\overline{g}}$, \\
and a coset of $(X-1)\cdot V^{\overline{f}}$ to a coset of $(X-1)\cdot V^{\overline{g}}$ .

It remains to prove that the particular coset $\tc(f)$ is taken by $\overline{h}$ to the particular coset $\tc(\overline{g})$.

This follows because for each $A\in\mathcal{E}$,\\ $\overline{h}\left( \overrightarrow{Af(A)}\right)= \overrightarrow{h(A)h\left(f(A)\right)}=\overrightarrow{h(A)g\left(h(A)\right)}.$

So $\overline{h}$ takes the element $\overrightarrow{Af(A)}$ of $\tc(f)$ to the element $\overrightarrow{h(A)g\left( h(A)\right)}$ of $\tc(g)$.
\end{proof}

Now the converse is also true:
\begin{theorem} 
Let $f,g:\mathcal{E}\longrightarrow\mathcal{E}$ be affine maps.\\
Assume $T:V\longrightarrow V$ is a linear isomorphism satisfying $T\circ\overline{f}\circ T^{-1}=\overline{g}$ (Otherwise put $T\circ \overline{f}=\overline{g}\circ T$ or $T:V^{\overline{f}}\longrightarrow V^{\overline{g}}$ is an isomorphism of modules over the ring $F[X]$).\\
Assume there exist $A,B\in\mathcal{E}$, such that $T\left(\overrightarrow{Af(A)}\right)=\overrightarrow{Bg(B)}$\\
Then:\\
There exists an affine isomorphism $h:\mathcal{E}\longrightarrow \mathcal{E}$ such that $h\circ f\circ h^{-1}=g$ and $\overline{h}=T$.\\
In particular $f\sim g$.
\end{theorem}

\begin{proof}
Remember that an affine map is determined by its linear part and by its action on a single point. We can therefore define an affine map $h:\mathcal{E}\longrightarrow\mathcal{E}$, by stipulating that $\overline{h}=T$ and that $h(A)=B$.\\
The affine map $h:\mathcal{E}\longrightarrow\mathcal{E}$ is an affine isomorphism because its linear part $\overline{h}:V\longrightarrow V$ is a linear isomorphism.\\
We assume that:
\begin{IEEEeqnarray*}{rCl}
T\left(\overrightarrow{Af(A)}\right)  &=& \overrightarrow{Bg(B)}\\
 \overline{h}\left(\overrightarrow{Af(A)}\right) &=& \overrightarrow{Bg(B)} \\
 \overrightarrow{h(A)h\left(f(A)\right)} &=& \overrightarrow{Bg(B)}\\
 \overrightarrow{Bh\left(f(A)\right)} &=& \overrightarrow{Bg(B)}
\end{IEEEeqnarray*}
From this we conclude that $h\left(f(A)\right)=g(B)$\\
that is: $h\left(f(A)\right)=g\left(h(A)\right)$\\
$h\circ f(A)=g\circ h(A)$.

Now look at the two affine maps $h\circ f\;,\;g\circ h:\mathcal{E}\longrightarrow \mathcal{E}$.\\
Their linear parts are equal:
\begin{IEEEeqnarray*}{rCl}
\overline{h\circ f}&=& \overline{h}\circ\overline{f}=T\circ\overline{f}=\overline{g}\circ T = \\
 &=& \overline{g}\circ\overline{h} = \overline{g\circ h}.
\end{IEEEeqnarray*}
Also they operate in the same way on the single point $A$:
\[h\circ f(A) = g\circ h(A)\]
From this we conclude that 
\[h\circ f = g\circ h\]
that is $h\circ f\circ h^{-1}=g$
\end{proof}

Let us rephrase the last two theorems:
\begin{theorem}{\label{theo:a}} 
Let $f,g:\mathcal{E}\longrightarrow\mathcal{E}$ be affine maps. Then $f\sim g$, if and only if, there is a linear isomorphism $T:V\longrightarrow V$ such that
\[T\circ \overline{f}\circ T^{-1} = \overline{g} \]
(Otherwise put:
$T:V^{\overline{f}}\longrightarrow V^{\overline{g}}$ is an isomorphism of modules over the ring $F[X]$) and $T$ takes the coset $\tc(f)$ to the coset $\tc(g)$
\end{theorem}

In our next section we shall introduce an invariant of cosets of the submodule $(X-1)\cdot V^{\overline{f}}$ which determines which coset of $(X-1)\cdot V^{\overline{f}}$
can be taken to which coset of $(X-1)\cdot V^{\overline{g}}$.

\section{The Algebraic Set-Up:} 

Let $V$ be a finite dimensional vector space over a field $F$. Let $L:V\longrightarrow V$ be a linear map. We shall study the $F[X]-\text{module} \; V^L$, and in particular, the cosets of the submodule $(X-1)\cdot V^L$. We shall study an invariant of cosets, called $\tau$, introduced for the first time in \cite{TartarYakir}.

Let $\text{min}_L(X)\in F[X]$ be the minimal (monic) polynomial of $L$.

Write
\[ \text{min}_L(X)=P_1(X)^{n_1}\cdot P_2(X)^{n_2}\cdot P_3(X)^{n_3}\cdot\ldots\cdot P_\ell (X)^{n_\ell},
\]
where for each $i\quad(1\leq i\leq \ell)$ the polynomial $P_i(X)$ is prime (and monic) and for each $i,j\quad (i\neq j)$
\[\text{gcd}\left(P_i(X),\;P_j(X)\right)=1.
\]
For each $i\quad(1\leq i\leq \ell)$ write 
\[V_i^L :=\text{ker}\left(P_i(L)^{n_i}\right).\]
Then $V^L=V_1^L\oplus V_2^L\oplus V_3^L\oplus \ldots\oplus V_\ell^L$ is the primary decomposition of $V^L$ as a module over the ring $F[X]$.
\begin{claim} 
For each $i\quad (1\leq i\leq\ell)$ \\
if $X-1\neq P_i(X)$ then $V_i^L\subseteq (X-1)\cdot V^L$.
\end{claim}
\begin{proof}
Find $R(X), S(X)\in F[X]$, \\
such that $R(X)\cdot (X-1)+S(X)\cdot P_i(X)^{n_i}=1$.

For each $\alpha\in V$
\begin{IEEEeqnarray*}{rCl}
\left(R(X)\cdot(X-1)+S(X)\cdot P_i(X)^{n_i}\right)\cdot\alpha&=&1\cdot\alpha\\
R(X)\cdot(X-1)\cdot\alpha + S(X)\cdot P_i(X)^{n_i}\cdot\alpha &=&\alpha
\end{IEEEeqnarray*}
But if $\alpha\in V_i^L$ then $P_i(X)^{n_i}\cdot\alpha=\overline{0}$.

Hence: $$ R(X)\cdot(X-1)\cdot\alpha=\alpha,$$
so $$ \alpha\in (X-1)\cdot V^L.$$
\end{proof}

The following claim  is the first in an effort to find "simple" representatives of cosets of $(X-1)\cdot V^L$ in $V^L$.
\begin{claim}\label{claim:a} 
If the polynomial $X-1$ does not divide the polynomial $\min\hbox{}_L(X)$, then $(X-1)\cdot V^L=V^L$.

If, however, $X-1\;\mid\;\min\hbox{}_L(X)$, say $X-1=P_1(X)$, then for each $\alpha\in V^L$ there exists an $\alpha '\in V_1^L$ such that
$$\alpha\equiv\alpha '\quad (\modp \; (X-1)\cdot V^L) $$
\end{claim}
\begin{proof}
If $X-1 \nmid \min\hbox{}_L(X)$, then for each $i\quad(1\leq i\leq\ell)$
\begin{IEEEeqnarray*}{lrCl}
          & V_i^L& \subseteq & (X-1)\cdot V^L,\\
\text{so}\quad & V^L& \subseteq & (X-1)\cdot V^L,\\
\text{so}\quad & V^L& = & (X-1)\cdot V^L.
\end{IEEEeqnarray*}
If $X-1\mid \min\hbox{}_L(X)$, say $X-1=P_1(X)$, then write $\alpha = \sum_{i=1}^{\ell}\alpha_i$ where for each $i\quad(1\leq i\leq\ell)\quad \alpha_i\in V^L_i$.\\
Write $\alpha':=\alpha_1$ and $\alpha'':=\sum_{i=2}^{\ell}\alpha_i$.\\
It follows that $\alpha=\alpha'+\alpha''$, $\alpha'\in V_1^L$ and $\alpha''\in(X-1)\cdot V^L$.
\end{proof}

We are going to introduce now the main tool in our study, the invariant $\tau(\alpha,L)$, an invariant which first appeared in \cite{TartarYakir}.
\begin{definition} 
Let $V$ be a finite dimensional vector space over a field $F$.\\
Let $L:V\longrightarrow V$ be a linear map.\\
For each $\alpha\in V$ define $\tau(\alpha,L)$ to be the smallest $k\in\mathbb{N}$, such that
$$\alpha\in\im(L-\id_V)+\ker\left(\left(L-\id_V\right)^k\right). $$
$(\mathbb{N}:=\{0,\,1,\,2,\,3,\ldots\})$
\end{definition}
\begin{remark} 
\begin{enumerate}
    \item For each $\alpha\in V$,\\
    $\tau(\alpha,L)=0$, if and only if, $\alpha\in (X-1)\cdot V^L$.
    \item If $X-1\; \nmid \;\min_L(X)$, then for each $\alpha\in V$, $\tau(\alpha,L)=0$
    \item If $X-1 \;\mid \;\min_L(X)$,\\
    say $X-1=P_1(X)$ (where $\min_L(X)=P_1(X)^{n_1}\cdot P_2(X)^{n_2}\cdot\ldots\cdot P_\ell(X)^{n_\ell}$)\\
    then for each $\alpha\in V,\quad \tau(\alpha,L)\leq n_1$.
    \item For each $\alpha,\;\alpha'\in V$, if $\alpha\equiv \alpha'\quad\left(\mod (X-1)\cdot V^L\right)$,\\
    then $\tau(\alpha, L)=\tau(\alpha', L)$.
\end{enumerate}
\end{remark}
\begin{definition} 
For each $\alpha\in V^L$, write $\langle\alpha\rangle:=\Sp_{F[X]}(\alpha)$.
In other words:
$$\langle\alpha\rangle:=\left\{R(X)\cdot\alpha\;\mid\;R(X)\in F[X]\right\}. $$
\end{definition}
\begin{claim}\label{claim:b} 
Assume $X-1=P_1(X)$ and that $\alpha\in V_1^L$ and $\beta\in\langle\alpha\rangle$. \\
Then either $\beta\in(X-1)\cdot V^L$ or $\langle\beta\rangle=\langle\alpha\rangle$.
\end{claim}
\begin{proof}
Put $\Ann(\alpha):=\left\{R(X)\in F[X]\;\mid\;R(X)\cdot\alpha=\overline{0}\right\}$.

Since $\alpha\in V_1^L$ it follows that $\left\langle P_1(X)^{n_1}\right\rangle\;\subseteq\;\Ann(\alpha)$.

This means that there is a $k\in\mathbb{N},\; k\leq n_1$, such that $\Ann(\alpha)=\left\langle P_1(X)^k\right\rangle$. 

Since $\beta\in\langle\alpha\rangle$ there is a polynomial $Q(X)\in F[X]$, such that $\beta =Q(X)\cdot\alpha$. 

If $P_1(X)\;\mid\;Q(X)$ then $\beta\in (X-1)\cdot V^L$. 

If $P_1(X)\;\nmid\; Q(X)$ then $\gcd\left(Q(X) \, ,\, P_1(X)^k\right)=1$.

Find $R(X), S(X)\in F[X]$, such that $R(X)\cdot Q(X) + S(X)\cdot P_1(X)^k=1$
\begin{IEEEeqnarray*}{rCl}
\left(R(X)\cdot Q(X)+S(X)\cdot P_1(X)^k\right)\cdot\alpha &=& \alpha\\
R(X)\cdot Q(X) \cdot\alpha + S(X)\cdot P_1(X)^k\cdot\alpha&=&\alpha
\end{IEEEeqnarray*}
\begin{IEEEeqnarray*}{lrCl}
\text{But}\quad &P_1(X)^k\cdot\alpha & = & \overline{0}\\
\text{So}\quad & R(X)\cdot Q(X)\cdot\alpha & = & \alpha\\
 &R(X)\cdot\beta & = &\alpha \\
 \text{So}\quad & &\alpha\in\langle\beta\rangle& 
\end{IEEEeqnarray*}\
and we conclude that $\langle\alpha\rangle=\langle\beta\rangle$.
\end{proof}
\begin{claim}\label{claim:c} 
Assume $X-1=P_1(X)$ and fix $V_1^L=\langle\gamma_1\rangle\oplus\langle\gamma_2\rangle\oplus\langle\gamma_3\rangle\oplus\cdots\oplus\langle\gamma_t\rangle$ any decomposition of $V_1^L$ into a direct sum of cyclic submodules.

Then, for each $\alpha\in V^L$ there exist $\alpha_1,\alpha_2,\alpha_3,\cdots,\alpha_t\in V_1^L$, such that
$$\alpha\equiv\sum_{i=1}^{t}\alpha_i\quad(\modp (X-1)\cdot V^L)  $$ 
and for each index $i\;(1\leq i \leq t)$ either $\alpha_i =\overline{0}$ or $\langle\alpha_i\rangle = \langle\gamma_i\rangle$.
\end{claim}
\begin{proof}
Let $\alpha\in V^L$. \\
Use first Claim \ref{claim:a} to find $\alpha'\in V_1^L$ such that $\alpha\equiv\alpha'\quad (\modp (X-1)\cdot V^L) $.

Write then $\alpha'$ as: $\alpha'=\alpha_1'+\alpha_2'+\alpha_3'+\ldots+\alpha_t'$ \\
where for each index $i\quad(1\leq i\leq t)\quad \alpha_i'\in\langle\gamma_i\rangle$.

Now, for each index $i\quad(1\leq i\leq t)$ define:
\begin{equation*}
\alpha_i :=
\begin{cases}
\overline{0} & \alpha_i'\in (X-1)\cdot V^L\\
\alpha_i' & \alpha_i'\notin (X-1)\cdot V^L
\end{cases}
\end{equation*}
Clearly $\sum_{i=1}^{k}\alpha_i\equiv\sum_{i=1}^{k}\alpha'_i\quad(\modp (X-1)\cdot V^L)$.

It is also clear (Claim \ref{claim:b}) that for each index $i\;(1\leq i\leq t)$ if $\alpha'_i\notin (X-1)\cdot V^L$, then $\langle\alpha_i\rangle=\langle\alpha'_i\rangle=\langle\gamma_i\rangle$.
\end{proof}

\begin{claim}{\label{claim:d}} 
Let $R$ be a ring, and let $W$ be a module over $R$.\\ 
Let $\varepsilon_1,\,\varepsilon_2,\,\varepsilon_3,\ldots ,\varepsilon_m \in W$.

Assume that $W=\langle\varepsilon_1\rangle\oplus\langle\varepsilon_2\rangle\oplus\langle\varepsilon_3\rangle\oplus\ldots\oplus\langle\varepsilon_m\rangle$.

Assume further that for each $i\;(1\leq i\leq m-1)$
$$\Ann(\varepsilon_i)\supseteq\Ann(\varepsilon_m) $$
Write $\varepsilon'_m:=\varepsilon_1 +\varepsilon_2 +\varepsilon_3 +\ldots +\varepsilon_m$

Then $\Ann(\varepsilon'_m)=\Ann(\varepsilon_m)$

and $W=\langle\varepsilon_1\rangle\oplus\langle\varepsilon_2\rangle\oplus\ldots\oplus\langle\varepsilon_{m-1}\rangle\oplus\langle\varepsilon'_m\rangle$.

Otherwise put: $\varepsilon'_m$ can be used instead of $\varepsilon_m$.
\end{claim}
\begin{proof}
$$\Ann(\varepsilon'_m)=\bigcap_{i=1}^{m}\Ann(\varepsilon_i)=\Ann(\varepsilon_m). $$
As to the second assertion, clearly
$$\varepsilon_1,\,\varepsilon_2,\ldots ,\varepsilon_{m-1},\,\varepsilon_m\,\in\, \langle\varepsilon_1\rangle +\langle\varepsilon_2\rangle +\dots +\langle\varepsilon_{m-1}\rangle +\langle\varepsilon'_m\rangle .$$
This implies that
$$W\;\subseteq\; \langle\varepsilon_1\rangle +\langle\varepsilon_2\rangle +\dots +\langle\varepsilon_{m-1}\rangle +\langle\varepsilon'_m\rangle$$
So
$$W\;=\; \langle\varepsilon_1\rangle +\langle\varepsilon_2\rangle +\dots +\langle\varepsilon_{m-1}\rangle +\langle\varepsilon'_m\rangle$$
It remains to show that the sum of submodules on the right, is
a direct sum.

Let $r_1,\, r_2,\ldots ,\,r_{m-1},\, s\in R$, such that $r_1\cdot\varepsilon_1 +r_2\cdot\varepsilon_2 +\ldots +r_{m-1}\cdot\varepsilon_{m-1} +s\cdot\varepsilon'_m =0 $.

$(r_1+s)\cdot\varepsilon_1+(r_2+s)\cdot\varepsilon_2+\ldots +(r_{m-1}+s)\cdot\varepsilon_{m-1}+s\cdot\varepsilon_m=\overline{0}$.

Hence, $s\cdot\varepsilon_m=\overline{0}$, \\
so $s\in\Ann(\varepsilon_m)$, \\
so for each $i\;(1\leq i\leq m-1)\; s\in\Ann(\varepsilon_i)$.

Also, for each $i\;(1\leq i\leq m-1)\;(r_i+s)\cdot\varepsilon_i=\overline{0}$, \\
so $r_i+s\in\Ann(\varepsilon_i)$, \\
so $r_i\in\Ann(\varepsilon_i)$.

To conclude: For each $i\;(1\leq i\leq m-1)\;r_i\in\Ann(\varepsilon_i)$ \\
so $r_i\cdot\varepsilon_i=\overline{0}$, \\
and also $s\in\Ann(\varepsilon_m)$ \\
(but $\Ann(\varepsilon_m)=\Ann(\varepsilon'_m)$) \\
so $s\cdot\varepsilon'_m=\overline{0}$.

This shows that the sum of submodules is indeed direct.
\end{proof}

The following claim will be the last in our quest of "simple" representatives $(\modp\;(X-1)\cdot V^L)$.

\begin{claim}{\label{claim:e}} 
Assume $X-1=P_1(X)$. For each $\alpha\in V^L$, there exists an $\alpha'\in V_1^L$, such that the cyclic submodule $\langle\alpha'\rangle$ is complemented in $V_1^L$, and such that
$$\alpha\equiv\alpha'\quad(\modp (X-1)\cdot V^L). $$
\end{claim}
\begin{proof}
Start with a fixed decomposition $V_1^L=\langle\gamma_1\rangle\oplus\langle\gamma_2\rangle\oplus\langle\gamma_3\rangle\oplus\dots\oplus\langle\gamma_t\rangle$ of the module $V_1^L$ into a direct sum of cyclic submodules.

Now, let $\alpha\in V^L$.

According to Claim \ref{claim:c} there exist $\alpha_1,\,\alpha_2,\ldots,\alpha_t\in V_1^L$, such that
$$\alpha\equiv\sum_{i=1}^{t}\alpha_i\quad\left(\modp\, (X-1)\cdot V_1^L\right) $$
and such that for each index $i\;(1\leq i\leq t)$ either $\alpha_i=\overline{0}$, or $\langle\alpha_i\rangle=\langle\gamma_i\rangle$.

Write $$J:=\left\{ i\;\mid \; 1\leq i\leq t,\;\alpha_i\neq\overline{0}\right\}$$
(If $J=\emptyset$ then $\alpha'=\overline{0}$ and we are done).

Write $$\alpha':=\sum_{i\in J}\alpha_i$$
(It is also true that $\alpha'=\sum_{i=1}^{t}\alpha_i$).

For each $i\in J$, $\alpha_i\in V_1^L$ so $$P_1(X)^{n_1}\in\Ann (\alpha_i).$$
Thus there exists a $d_i\;(1\leq d_i\leq n_1)$ such that $$\Ann(\alpha_i)=\left\langle P_1(X)^{d_i}\right\rangle.$$
Fix an index $m\in J$ such that for each $i\in J,\; d_m\geq d_i$.
Thus for each $i\in J$, $$\Ann(\alpha_m)\subseteq\Ann (\alpha_i).$$
We plan to use Claim \ref{claim:d}.
Let $$W:=\bigoplus_{i\in J}\langle \alpha_i\rangle.$$
Write $$J':=J\smallsetminus\{m\}.$$
According to Claim \ref{claim:d}, the decomposition $$W=\bigoplus_{i\in J}\langle \alpha_i\rangle$$ can be replaced with another decomposition $$W=\left(\bigoplus_{i\in J'}\langle\alpha_i\rangle\right)\oplus\langle\alpha'\rangle.$$
Now, the submodule $\langle\alpha'\rangle$ is complemented in $W$, and $W$ is complemented in $V_1^L$, so $\langle\alpha'\rangle$ is complemented in $V_1^L$.
\end{proof}
\begin{definition} 
Assume $X-1=P_1(X)$. Let $\alpha\in V^L$. According to Claim \ref{claim:e} there is an $\alpha'\in V_1^L$, such that $\alpha\equiv\alpha'\;\left(\mod\;(X-1)\cdot V^L\right)$ and such that the cyclic submodule $\langle\alpha'\rangle$ is complemented in $V_1^L$. Call $\alpha'$ a simple representative of $\alpha$.
\end{definition}
We shall make now a short excursion into the study of cyclic submodules of $V^L$.

We have (as usual) $V$, a finite dimensional vector space over $F$. \\
Let $L:V\longrightarrow V$ be a linear map. For each $\alpha\in V$, denote by $\langle\alpha\rangle$ the cyclic submodule of $V^L$ generated by $\alpha$. Thus
$$\langle\alpha\rangle := \left\{R(X)\cdot\alpha\,\mid\,R(X)\in F[X]\right\}. $$
\begin{claim}\label{claim:f} 
Let $L:V\longrightarrow V$ be a linear map. Fix a polynomial $P(X)\in F[X]$, such that $\deg\left(P(X)\right)=1$. Then:
\begin{enumerate}
    \item For every $F-$subspace $U$ of $V$, $U$ is invariant under $L$, if and only if, $U$ is invariant under $P(L)$.
    \item For every $\alpha\in V^L$, the cyclic submodule $\langle\alpha\rangle$ is spanned (as a vector space over $F$) by the infinite sequence
    $$\alpha,\,P(X)\cdot\alpha,\,P(X)^2\cdot\alpha,\,P(X)^3\cdot\alpha,\,P(X)^4\cdot\alpha,\ldots $$
    \item Let $P(X)^d\cdot\alpha$ be the first non-pivotal element in the sequence above. (That is the first element which is spanned over $F$ by previous elements). Then the linear dependence (over $F$):
    $$\sum_{i=0}^{d-1}a_i\left(P(X)^i\cdot\alpha\right)=P(X)^d\cdot\alpha , $$
    when multiplied by $P(X)^j$, shows that for each $j\;(j\in\mathbb{N})$ the element $P(X)^{d+j}\cdot\alpha$ is non-pivotal, so $\dim\left(\langle\alpha\rangle\right)=d$.
    \item Denote by $L'$, the restriction of $L$ to $\langle\alpha\rangle$, that is 
    $$L':=L\upharpoonright\langle\alpha\rangle . $$
    Then the linear dependence above gives rise to $\min_{L'}(X)$, the minimal polynomial of $L'$.
    $$\min\hbox{}_{L'}(X)=\sum_{i=0}^{d-1}(-a_i)\cdot P(X)^i + P(X)^d$$
\end{enumerate}
\end{claim}
\begin{proof}
An easy exercise.
\end{proof}
\begin{claim}{\label{claim:g}} 
Assume $X-1=P_1(X)$. Let $\gamma\in V_1^L$ and assume that the cyclic submodule $\langle\gamma\rangle$ is complemented in $V_1^L$. Then
$$\Ann(\gamma)=\left\langle P_1(X)^{\tau (\gamma ,L)}\right\rangle . $$
\end{claim}
\begin{proof}
Since $\gamma\in V_1^L ,\, P_1(X)^{n_1}\cdot\gamma=\overline{0}$, so there exists $d\in\mathbb{N},\, d\leq n_1$, such that
$$\Ann(\gamma)=\langle P_1(X)^d\rangle . $$
\end{proof}
From Claim \ref{claim:f} it follows that the sequence
$$\gamma,\,P_1(X)\cdot\gamma,\,P_1(X)^2\cdot\gamma, \ P_1(X)^3\cdot\gamma,\ldots,P_1(X)^{d-1}\cdot\gamma $$
is a basis for the cyclic submodule $\langle\gamma\rangle$ (regarded as a vector space over $F$).

We regard the map
$$P_1(L)\upharpoonright\langle\gamma\rangle\; :\;\langle\gamma\rangle\longrightarrow\langle\gamma\rangle $$
Let us study its kernel and its image.

Clearly
$$P_1(X)^{d-1}\cdot\gamma\in\ker(P_1(L)\upharpoonright\langle\gamma\rangle) $$
(so $\dim\left(\ker\left(P_1(L)\upharpoonright\langle\gamma\rangle\right)\right)\geq 1$) and clearly
$$P_1(X)\cdot\gamma,\,P_1(X)^2\cdot\gamma, \ P_1(X)^3\cdot\gamma,\ldots,P_1(X)^{d-1}\cdot\gamma\in\im\left(P_1(L)\upharpoonright\langle\gamma\rangle\right) $$
so $$\dim\left(\im\left(P_1(L)\upharpoonright\langle\gamma\rangle\right)\right) \geq d-1 .$$
But since 
$$\dim\left(\ker\left(P_1(L)\upharpoonright\langle\gamma\rangle\right)\right) + \dim\left(\im\left(P_1(L)\upharpoonright\langle\gamma\rangle\right)\right) =d ,$$
it follows that
\begin{IEEEeqnarray*}{rCll}
\dim\left(\ker\left(P_1(L)\upharpoonright\langle\gamma\rangle\right)\right) & = & 1 & \text{and}\\
\dim\left(\im\left(P_1(L)\upharpoonright\langle\gamma\rangle\right)\right) & = & d-1 & 
\end{IEEEeqnarray*}
In particular
$$\im\left(P_1(L)\upharpoonright\langle\gamma\rangle\right) = \Sp\left(P_1(X)\cdot\gamma,\,P_1(X)^2\cdot\gamma,\, P_1(X)^3\cdot\gamma,\ldots,P_1(X)^{d-1}\cdot\gamma\right). $$
Our assumption is that the cyclic submodule $\langle\gamma\rangle$ is complemented in $V_1^L$. Remember that the primary submodule $V_1^L$ is complemented in the module $V^L$. Therefore $\langle\gamma\rangle$ is complemented in $V^L$. There exists therefore an $F[X]-$submodule $W$ of $V^L$, such that $$\langle\gamma\rangle\oplus W= V^L .$$
It follows that
$$\im\left(P_1(L)\right)\subseteq \Sp\left(P_1(X)\cdot\gamma,\,P_1(X)^2\cdot\gamma,\, P_1(X)^3\cdot\gamma,\ldots,P_1(X)^{d-1}\cdot\gamma\right)+W $$
Also:
$$\ker\left(P_1(L)^{d-1}\right)\subseteq \Sp\left(P_1(X)\cdot\alpha,\,P_1(X)^2\cdot\alpha,\, P_1(X)^3\cdot\alpha\ldots,P_1(X)^{d-1}\cdot\gamma\right)+W$$
So $\gamma\notin\im\left(P_1(L)\right)+\ker\left(P_1(L)^{d-1}\right)$ and therefore $\tau(\gamma,L)>d-1$. 

But
$$\Ann(\gamma)=\left\langle P_1(X)^d\right\rangle, $$
so
$$\gamma\in\ker\left(P_1(L)^d\right) $$
Thus
$$\tau(\gamma,L)=d . $$
Therefore
$$\Ann(\gamma)=\left\langle P_1(X)^d\right\rangle = \left\langle P_1(X)^{\tau(\gamma,L)}\right\rangle .$$

And now to our main algebraic result:
\begin{theorem}{\label{theo:b}} 
Let $V$ be finite dimensional vector space over a field $F$.
Let $L,M:V\longrightarrow V$ two linear maps. Assume that $L\sim M$ (L is similar to M) which means of course that the $F[X]-$modules $V^L$ and $V^M$ are isomorphic.

Then, for every $\alpha,\beta\in V$, \\
there exists an $F[X]-$isomorphism $T:V^L\longrightarrow V^M$,\\ 
taking the coset $\alpha+(X-1)\cdot V^L$ to the coset $\beta+(X-1)\cdot V^M$, \\
if and only if $\tau(\alpha,L)=\tau(\beta,M)$.
\end{theorem}
\begin{proof}
\begin{enumerate}
    \item Assume first that there exists an $F[X]-$isomorphism\\ 
    $T:V^L\longrightarrow V^M$ which takes the coset $\alpha + (X-1)\cdot V^L$ to the coset $\beta +(X-1)\cdot V^M $.
    For every $k\in\mathbb{N}$, if $\tau(\alpha,L)\leq k$ then
    \begin{IEEEeqnarray*}{rCl}
    \alpha & \in & (X-1)\cdot V^L+\ker\left(\left(L-\id_V\right)^k\right),\\
    T(\alpha) & \in & (X-1)\cdot V^M+\ker\left(\left(M-\id_V\right)^k\right),\\
    \beta & \in & (X-1)\cdot V^M+\ker\left(\left(M-\id_V\right)^k\right)
    \end{IEEEeqnarray*}
    (because $T(\alpha)\equiv\beta \quad(\modp(X-1)\cdot V^M)$).
    
    So $\tau(\beta,M)\leq k$.
    
    This shows that $\tau(\beta,M)\leq\tau(\alpha,L)$.\\
    The same proof using $T^{-1}:V^M\longrightarrow V^L$, \\
    will show that $\tau(\alpha,L)\leq\tau(\beta, M).$
    \item Conversely, assume that $\tau(\alpha,L)=\tau(\beta, M)$. We assume of course that $V^L\cong V^M$. Now, if $\tau(\alpha, L)=0$ (and then $\tau(\beta, M)=0$) then
    $$ \alpha\in(X-1)\cdot V^L\;,\;\beta\in(X-1)\cdot V^M.$$
    Every $F[X]-$isomorphism $T:V^L\longrightarrow V^M$,\\
    takes the submodule $(X-1)\cdot V^L$ to the submodule $(X-1)\cdot V^M$.
    
    Assume now that $\tau(\alpha,L)\geq 1$. Choose $\alpha'\in V_1^L$, a simple representative of $\alpha$. That is 
    $$\alpha\equiv\alpha'\quad(\modp(X-1)\cdot V^L) $$
    and $\langle\alpha'\rangle$ is a complemented submodule in $V_1^L$.
    
    In the same manner, choose $\beta'\in V_1^M$, a simple representative of $\beta$. That is 
    $$\beta\equiv\beta'\quad(\modp (X-1)\cdot V^M) $$
    and $\langle\beta'\rangle$ is a complemented submodule in $V_1^M$.
    
    Now (using Claim \ref{claim:g})
    \begin {IEEEeqnarray*}{rCl}
    \Ann(\alpha') & = &\left\langle P_1(X)^{\tau(\alpha',L)}\right\rangle\\
     & = & \left\langle P_1(X)^{\tau(\alpha, L)} \right\rangle.
    \end{IEEEeqnarray*}
    Also 
    \begin {IEEEeqnarray*}{rCl}
    \Ann(\beta') & = &\left\langle P_1(X)^{\tau(\beta',M)}\right\rangle\\
     & = & \left\langle P_1(X)^{\tau(\beta, M)} \right\rangle.
    \end{IEEEeqnarray*}
    Our assumption is that
    $$\tau(\alpha,L)=\tau(\beta, M)\quad\text{so}\quad
    \Ann(\alpha')=\Ann(\beta').$$
    
    We shall construct an isomorphism $T:V^L\longrightarrow V^M$ \\
    such that $T(\alpha')=\beta'$.
    
    How shall we do it?
    
    Decompose $V^L$ into a direct sum of primary components. (This is a unique decomposition).
    
    Decompose then each primary component into a direct sum of cyclic submodules (this is not unique) taking care that one of the cyclic submodules of $V_1^L$ will be $\langle\alpha'\rangle$. This can be done because the submodule $\langle\alpha'\rangle$ is complemented in the primary submodule $V_1^L$.
    
    In the same manner decompose the module $V^M$ into a direct sum of primary components and then decompose each primary component into a direct sum of cyclic submodules, taking care that one of the cyclic submodules of $V_1^M$ will be $\langle\beta'\rangle$. This can be done because the submodule $\langle\beta'\rangle$ is complemented in the primary submodule $V_1^M$.
    
    All in all we have a decomposition
    \begin{IEEEeqnarray*}{rCl}
     V^L & = & \bigoplus_{i\in Y}\langle\alpha_i\rangle\\
     V^M & = & \bigoplus_{i\in Y}\langle\beta_i\rangle
    \end{IEEEeqnarray*}
    taking care that for each $i\in Y$,
    $$\Ann(\alpha_i)=\Ann(\beta_i) $$
    and taking care also that for the same index $i_0\in Y$,
    $$\alpha_{i_0}=\alpha'\; ,\; \beta_{i_0}=\beta' .$$
    We then define $T:V^L\longrightarrow V^M$, by stipulating that for every $i\in Y$,
    $$T(\alpha_i):=\beta_i $$
    (and in particular $T(\alpha_{i_0})=\beta_{i_0}$ that is $T(\alpha')=\beta'$).
\end{enumerate}
\end{proof}

\section{Back to Geometry} 

Let $(\mathcal{E},V)$ be an affine space over a field $F$. We assume that $V$ is finite dimensional.

For an affine map $f:\mathcal{E}\longrightarrow\mathcal{E}$, we defined the trajectory coset of $f$ to be
$$\tc(f)=\left\{\overrightarrow{Af(A)}\;\mid\; A\in \mathcal{E}\right\}. $$
We proved that $\tc(f)$ is indeed a coset of the submodule $(X-1)\cdot V^{\overline{f}}$ of $V^{\overline{f}}$.

\begin{definition} 
Let $f:\mathcal{E}\longrightarrow\mathcal{E}$ be an affine map. Choose $A\in \mathcal{E}$ and define
$$\tau(f):=\tau\left(\overrightarrow{Af(A)},\overline{f}\right). $$
\end{definition}
Notice that the definition does not depend on the choice of the point $A\in\mathcal{E}$. If $B\in\mathcal{E}$ then 
$$\overrightarrow{Af(A)}\equiv\overrightarrow{Bf(B)}\;(\mod (X-1)\cdot V^{\overline{f}}) . $$
Notice also that $\tau(f)$ is the smallest $k\in\mathbb{N}$, such that
$$\overrightarrow{Af(A)}\in (X-1)\cdot V^{\overline{f}} + \ker\left(\left(\overline{f}-\id_V\right)^k\right) . $$
Our main theorem is:
\begin{theorem} 
Let $(\mathcal{E},V)$ be an affine space over a field $F$. Assume $V$ is finite dimensional.

Then, for every affine maps $f,g:\mathcal{E\longrightarrow\mathcal{E}}$, $f\sim g$ if and only if, the linear maps $\overline{f},\overline{g}:V\longrightarrow V$ are similar (we write it $\overline{f}\sim\overline{g}$) and $\tau(f)=\tau(g)$.
\end{theorem}
\begin{proof}
From Theorem \ref{theo:a} it follows that $f\sim g$, if and only if, there is a linear isomorphism (of $F[X]-$modules)
$$T:V^{\overline{f}}\longrightarrow V^{\overline{g}} $$
which takes the coset $\tc(f)$ to the coset $\tc(g)$.

According to Theorem \ref{theo:b}, such a linear isomorphism exists, if and only if, 
$$\tau\left(\overrightarrow{Af(A)},\overline{f}\right)=\tau\left(\overrightarrow{Bg(B)},\overline{g}\right) $$
(where $A,B$ are any points of $\mathcal{E}$)
\end{proof}

\section{Invariant Flats} 

We shall study now invariant flats of affine maps, in order to give $\tau$ another meaning of a geometrical nature.

\begin{claim}{\label{claim:h}} 
Let $(\mathcal{E},V)$ be an affine space over a field $F$.
Let $f:\mathcal{E}\longrightarrow\mathcal{E}$ be an affine map. 

Then for every $A\in\mathcal{E}$ and every $W$ a subspace of $V$, the flat $A+W$ is invariant under $f$, if and only if, $\overrightarrow{Af(A)}\in W$ and $W$ is invariant under $\overline{f}$.
\end{claim}
\begin{proof}
Assume first that the flat $A+W$ is invariant under $f$.

In particular, $A\in A+W$ so $f(A)\in A+W$ so $\overrightarrow{Af(A)}\in W$.

Also, $A+W=f(A)+W$.

Now let $\alpha\in W$. Since $A+\alpha\in A+W$, it follows that $f(A+\alpha)\in A+W$.

So $f(A+\alpha)\in f(A)+W$.

That is $f(A)+\overline{f}(\alpha)\in f(A)+W$.

This implies that $\overline{f}(\alpha)\in W$.

Conversely, assume that $\overrightarrow{Af(A)}\in W$ and that $W$ is invariant under $\overline{f}$. Let $B\in A+W$. We have to show that $f(B)\in A+W$.

Our assumption that $B\in A+W$, implies that $\overrightarrow{AB}\in W$. Thus $\overline{f}\left(\overrightarrow{AB}\right)\in W$ (because $W$ is invariant under $\overline{f}$).
Now,
\begin{IEEEeqnarray*}{rCl}
\overrightarrow{Af(B)} &=& \overrightarrow{Af(A)} + \overrightarrow{f(A)f(B)}\\
 &=&\overrightarrow{Af(A)}+\overline{f}\left(\overrightarrow{AB}\right).
\end{IEEEeqnarray*}
$\overrightarrow{Af(A)}, \overline{f}\left(\overrightarrow{AB}\right)\in W$, so 
$\overrightarrow{Af(A)}+ \overline{f}\left(\overrightarrow{AB}\right)\in W$, so
$\overrightarrow{Af(B)}\in W$, so
$f(B)\in A+W$
\end{proof}
\begin{definition} 
Let $f:\mathcal{E}\longrightarrow\mathcal{E}$ be an affine map and let $A\in\mathcal{E}$.

Denote by $\Flat(A,f)$ the flat of $\mathcal{E}$ which passes throught the point $A$ and whose direction is $\left\langle\overrightarrow{Af(A)}\right\rangle$, the cyclic $F[X]-$submodule of $V^{\overline{f}}$, generated by the vector $\overrightarrow{Af(A)}$.

In other words:
$$\Flat(A,f):=A+\left\langle\overrightarrow{Af(A)}\right\rangle . $$
\end{definition}
\begin{claim} 
Let $f:\mathcal{E}\longrightarrow\mathcal{E}$ be an affine map. Let $A\in\mathcal{E}$. Then:
\begin{enumerate}
    \item $A\in\Flat(A,F)$
    \item The flat $\Flat(A,f)$ is invariant under $f$.
    \item If $\mathcal{F}$ is a flat in $\mathcal{E}$ which is invariant under $f$ and which passes throught the point $A$, then $$\Flat(A,f)\subseteq\mathcal{F}. $$
    \item Denote by $\overline{f}'$ the restriction of $\overline{f}$ to $\left\langle\overrightarrow{Af(A)}\right\rangle$.
    That is:
    $$\overline{f}'=\overline{f}\upharpoonright
    \left\langle\overrightarrow{AfA}\right\rangle .$$
    Then:
    $$\dim\left(\Flat(A,f)\right)=\deg\left(\min\hbox{}_{\overline{f}'}(X)\right). $$
    That is, the dimension of the flat $\Flat(A,f)$ is the degree of the minimal polynomial of the linear map $\overline{f}':\left\langle\overrightarrow{Af(A)}\right\rangle\longrightarrow\left\langle\overrightarrow{Af(A)}\right\rangle$.
\end{enumerate}
\end{claim}
\begin{proof}
\begin{enumerate}
    \item Follows from the definition
    \item Claim \ref{claim:h}
    \item Claim \ref{claim:h}
    \item Follows from part 4. of Claim \ref{claim:f}.
\end{enumerate}
\end{proof}
\begin{theorem}{\label{theo:d}} 
Let $(\mathcal{E},V)$ be an affine space over a field $F$. We assume that $V$ is finite dimensional.
Let $f:\mathcal{E}\longrightarrow\mathcal{E}$ be an affine map. 

From the coset $\tc(f)$ (a coset of the submodule $(X-1)\cdot V^{\overline{f}}$ in the module $V^{\overline{f}}$) choose a simple representative $\alpha$. By this I mean:\\
If $\tau(f)=0$ (so $\tc(f)=(X-1)\cdot V^{\overline{f}}$) choose $\alpha:=\overline{0}$. \\
If $\tau(f)\geq 1$, choose an $\alpha\in V_1^{\overline{f}}$ such that the cyclic submodule $\left\langle\alpha\right\rangle$ is complemented in $V_1^{\overline{f}}$ (see Claim \ref{claim:e}).

After you chose the simple representative $\alpha$ from the coset $\tc(f)$, choose a point $A\in\mathcal{E}$, such that $\alpha=\overrightarrow{Af(A)}$. Then:
$$\dim\left(\Flat(A,f)\right)=\tau(f), $$
so we found a flat of $\mathcal{E}$ which is invariant under $f$, passes through the point $A$, and is of dimension $\tau(f)$. Its direction is a complemented submodule in $V_1^{\overline{f}}$
\end{theorem}
\begin{proof}
If $\tau(f)=0$ then $\alpha=\overline{0}$ and $\overrightarrow{Af(A)}=\overline{0}$ so $f(A)=A$.

$\Flat(A,f)=\{A\}$ and indeed in this case $\Flat(A,f)=A+\langle\overline{0}\rangle$.

Assume now that $\tau(f)\geq 1$. Write $k:=\tau(f)$. So we chose $\alpha\in\tc(f)$ (and $A\in\mathcal{E}$ such that $\overrightarrow{Af(A)}=\alpha$) such that $\alpha\in V_1^{\overline{f}}$, and the cyclic submodule $\langle\alpha\rangle$ is complemented in $V_1^{\overline{f}}$.

It follows from Claim \ref{claim:g} that $\Ann(\alpha)=\left\langle P_1(X)^{\tau(\alpha,\overline{f})}\right\rangle$.

That is $\Ann(\alpha)=\left\langle(X-1)^k\right\rangle$.

So, again, if we denote by $\overline{f}'$ the restriction of $\overline{f}$ to $\langle\alpha\rangle$, then 
$$\min\hbox{}_{\overline{f}'}(X)=(X-1)^k.$$
Thus $\dim\left(\langle\alpha\rangle\right)=k$ (part~4 of Claim~\ref{claim:f}).

Remember that $\Flat(A,f)=A+\left\langle\overrightarrow{Af(A)}\right\rangle = A+\langle\alpha\rangle$.

So $\dim(\Flat(A,f))=k$.
\end{proof}
\begin{theorem}{\label{theo:e}} 
Let $f:\mathcal{E}\longrightarrow\mathcal{E}$ be an affine map.

Then any flat of $\mathcal{E}$ which is invariant under $f$, has dimension $\geq\tau(f)$.
\end{theorem}
\begin{proof}
It suffices to prove our assertion for flats of the form $\Flat(A,f)$ (because every flat which is invariant under $f$ contains such a flat).

But $\Flat(A,f)=A+\left\langle\overrightarrow{Af(A)}\right\rangle$. \\
It suffices therefore to prove that for every $A\in\mathcal{E}$,
$$\dim\left(\left\langle\overrightarrow{Af(A)}\right\rangle\right) \geq \tau(f).$$
Assume (using the notation at the Algebraic Set-Up) that 
$$\min\hbox{}_{\overline{f}}(X)=P_1(X)^{n_1}\cdot P_2(X)^{n_2}\cdot P_3(X)^{n_3}\cdot\ldots\cdot P_\ell(X)^{n_\ell},$$
where $P_1(X)=X-1$. (Remember that if $X-1\nmid\min_{\overline{f}}(X)$) then $\tau(f)=0$ and we are done).

For each $i\;(1\leq i\leq\ell)$ write $$V_i^{\overline{f}}:=\ker\left(P_i(\overline{f})^{n_i}\right) $$
and remember that
$$V^{\overline{f}}=V_1^{\overline{f}}\oplus V_2^{\overline{f}}\oplus V_3^{\overline{f}}\oplus \ldots\oplus V_\ell^{\overline{f}}. $$
Now let $A\in\mathcal{E}$. Write $\alpha:=\overrightarrow{Af(A)}$.
We have to show that
$$\dim\left(\langle\alpha\rangle\right)\geq\tau(f). $$
Write $\alpha=\beta+\gamma$ where $\beta\in V_1^{\overline{f}}$ and $\gamma\in\oplus_{i=2}^{\ell}V_i^{\overline{f}}$.
$$ \Ann(\alpha)=\Ann(\beta)\cap\Ann(\gamma)$$
So if
\begin{IEEEeqnarray*}{rCl}
\overline{f}' & := & \overline{f}\upharpoonright \langle\alpha\rangle\\
\overline{f}'' & := & \overline{f}\upharpoonright \langle\beta\rangle\\
\overline{f}'''& := & \overline{f}\upharpoonright \langle\gamma\rangle,
\end{IEEEeqnarray*}
then
\begin{IEEEeqnarray*}{rCl}
\min\hbox{}_{\overline{f}'}(X) & = & \lcm\left(\min\hbox{}_{\overline{f}''}(X)\, ,\,\min\hbox{}_{\overline{f}'''}(X)\right)\\
 & = & \min\hbox{}_{\overline{f}''}(X) \cdot\min\hbox{}_{\overline{f}'''}(X).
\end{IEEEeqnarray*}
Notice that $\gamma\in(X-1)\cdot V^{\overline{f}}$\\
so $\alpha\equiv\beta\;(\mod(X-1)\cdot V^{\overline{f}})$\\
so $\tau\left(\alpha,\overline{f}\right)=\tau\left(\beta,\overline{f}\right)$.\\
Now, since $\beta\in V_1^{\overline{f}}$, it follows that
$$\min\hbox{}_{\overline{f}''}(X)=P_1(X)^d $$
where
$$\tau\left(\beta,\overline{f}\right)\leq d\leq n_1 $$
But
$$ \tau\left(\beta,\overline{f}\right)=\tau\left(\alpha,\overline{f}\right)$$
so
$$\tau\left(\alpha,\overline{f}\right)\leq d .$$
\begin{IEEEeqnarray*}{rCl}
\min\hbox{}_{\overline{f}'}(X) & = & \min\hbox{}_{\overline{f}''}(X)\cdot\min\hbox{}_{\overline{f}'''}(X)\\
  & = & P_1(X)^d\cdot\min\hbox{}_{\overline{f}'''}(X)
\end{IEEEeqnarray*}
Thus
$$\deg\left(\min\hbox{}_{\overline{f}'}(X)\right)\geq d $$
So
$$\deg\left(\min\hbox{}_{\overline{f}'}(X)\right)\geq \tau\left(\alpha,\overline{f}\right). $$
Which means $$\dim\left(\langle\alpha\rangle\right)\geq\tau(f). $$
\end{proof}

The Invariance Level:

In \cite{Tarrida} the notion of invariance level of an affine map $f$ is defined and is denoted by $\rho(f)$.
\begin{definition} 
Let $f:\mathcal{E}\longrightarrow\mathcal{E}$ be an affine map. If $\mathcal{E}$ has a flat, which is invariant under $f$, of dimension $k$, but not one of smaller dimension, then we say that the invariance level of $f$ is $k$. We write $\rho(f)=k$.
\end{definition}

Our Theorem \ref{theo:d} and Theorem \ref{theo:e} show:
\begin{theorem}{\label{theo:f}} 
For every affine map
$$f:\mathcal{E}\longrightarrow\mathcal{E},\quad\rho(f)=\tau(f). $$
\qed
\end{theorem}
We can do better than Theorem~\ref{theo:e}. In Theorem~\ref{theo:e} we show that every flat which is invariant under $f$, is of dimension $\geq\tau(f)$.

One can actually show that if $\mathcal{F}$ is a flat which is invariant under $f$, and $\dim(\mathcal{F})=\tau(f)$, then $\dir(\mathcal{F})\subseteq V_1^{\overline{f}}$.

By looking closely at the decomposition $\alpha=\beta + \gamma$ in the proof of Theorem~\ref{theo:e}, one can show that not only 
$\dim\left(\right\langle\alpha\rangle)\geq\dim\left(\langle\beta\rangle\right)$, but also that $\langle\alpha\rangle\supseteq\langle\beta\rangle$ (and $\langle\alpha\rangle\supseteq\langle\gamma\rangle$). We state it as:
\begin{theorem}{\label{theo:g}} 
Assume $X-1\,\mid\,\min\hbox{}_{\overline{f}}(X)$, say $X-1=P_1(X)$.

Let $\mathcal{F}$ be a flat of $\mathcal{E}$ which is invariant under $f$.

If $\dim(\mathcal{F})=\tau(f)$, then $\dir(\mathcal{F})\subseteq V_1^{\overline{f}}$.
\end{theorem}
\begin{proof}
Let
\[ \min\hbox{}_{\overline{f}}(X)=P_1(X)^{n_1}\cdot P_2(X)^{n_2}\cdot P_3(X)^{n_3}\cdot\ldots\cdot P_\ell(X)^{n_\ell}\]
where for each index $i\;(1\leq i\leq\ell)$ the polynomial $P_i(X)$ is prime and (monic) and for each $i,j\;(i\neq j)$
\[\gcd\left(P_i(X),P_j(X)\right)=1\]
Assume $P_1(X)=X-1$.

Let $Q_1(X),\,Q_2(X),\,\ldots ,Q_\ell(X)$ be the sequence of "Lagrange interpolants" associated with the sequence
$P_1(X)^{n_1}\cdot P_2(X)^{n_2}\cdot P_3(X)^{n_3}\cdot\ldots\cdot P_\ell(X)^{n_\ell}$.

Remember that in the ring $F[X]$:
\[Q_j(X)\equiv 1\quad(\modp P_j(X)^{n_j}).\]
and if $i\neq j$ then
\[Q_j(X)\equiv 0\quad(\modp P_i(X)^{n_i}).\]
\end{proof}
The sequence $Q_1(X),\,Q_2(X),\,\ldots ,Q_\ell(X)$ \\
is uniquely determined $\left(\modp\;\min\hbox{}_{\overline{f}}(X)\right)$.

Remember also that
\[Q_1(X)+Q_2(X)+\ldots +Q_\ell(X)\equiv 1\quad \left(\modp\;\min\hbox{}_{\overline{f}}(X)\right).\]
Now, let $\mathcal{F}$ be a flat of $\mathcal{E}$ which is invariant under $f$.

Assume \[\dim(\mathcal{F})=\tau(f).\]

Clearly, there is a point $A\in\mathcal{E}$, such that $\mathcal{F}=\Flat(A,f).$

Write $\alpha := \overrightarrow{Af(A)}.$
\[\mathcal{F}=A+\langle\alpha\rangle .\]
\begin{IEEEeqnarray*}{rCl}
\alpha &=& 1\cdot \alpha = \left(Q_1(X)+Q_2(X)+\ldots +Q_\ell(X)\right)\cdot\alpha\\
\alpha &=& Q_1(X)\cdot\alpha +Q_2(X)\cdot\alpha +\ldots +Q_{\ell}(X)\cdot\alpha.
\end{IEEEeqnarray*}
For each index $i\;(1\leq i\leq\ell)$
\[Q_i(X)\cdot\alpha\in V_i^{\overline{f}}. \] 
Write $\beta := Q_1(X)\cdot\alpha$\\
and $\gamma := Q_2(X)\cdot\alpha + Q_3(X)\cdot\alpha + \ldots + Q_\ell(X)\cdot \alpha$.\\
Thus $\alpha=\beta+\gamma$.\\
Since $\gamma\in (X-1)\cdot V^{\overline{f}}$ it follows that 
\[\alpha\equiv \beta\quad\left(\modp\;(X-1)\cdot V^{\overline{f}}\right) \]
and therefore $\beta\in \tc(f)$.

There exists a point $B\in\mathcal{E}$ such that $\beta =\overline{Bf(B)}$.

Notice $\Flat(B,f)=B+\langle\beta\rangle$.

The flat $B+\langle\beta\rangle$ is invariant under $f$.

Therefore $\dim\left(\langle\beta\rangle\right) 
\geq\tau(f)$.

But $\beta=Q_1(X)\cdot\alpha$\\
so $\langle\beta\rangle\subseteq\langle\alpha\rangle$.

From the assumption that $\dim\left(\langle\alpha\rangle\right)=\tau(f)$, \\
it follows that
\[\langle\beta\rangle=\langle\alpha\rangle\]
Thus
\[\dir(\mathcal{F})=\langle\beta\rangle\]
and
\[\langle\beta\rangle\subseteq V_1^{\overline{f}}\]
\qed

The moral of our journey, in conclusion, is that the study of $(X-1)\cdot V_1^{\overline{f}}$, is crucial for understanding questions such as similarity and invariance level.

\begin{remark}
In an independant work \cite{Hou} Xiang-dong Hou finds a system of distinct representatives of the conjugacy class of $AGL(m,F)$
\end{remark}


\printbibliography

\end{document}